\documentclass{article}
\usepackage{latexsym}
\usepackage{a4wide}
\usepackage{amsmath, rotating, color}
\usepackage{amsfonts,amssymb, amsthm}
\usepackage{pictexwd}
\usepackage{bbm}

\newcommand{\be}{\begin{enumerate}}
\newcommand{\ee}{\end{enumerate}}
\newcommand{\bi}{\begin{itemize}}
\newcommand{\ei}{\end{itemize}}

\newcommand\unnumberedfootnote[1]{ %
        \let\temp=\thefootnote %
        \renewcommand{\thefootnote}{}%
        \footnote{#1}%
        \let\thefootnote=\temp%
        \addtocounter{footnote}{-1}}
\newtheorem{theorem}{Theorem}
\newtheorem{proposition}{Proposition}[section]
\newtheorem{lemma}[proposition]{Lemma}

\newtheorem{definition}[proposition]{Definition}
\theoremstyle{definition}
\newtheorem{remark}[proposition]{Remark}

\newtheorem{step}{Step}

\numberwithin{equation}{section}

\begin{document}


\title{\LARGE The Aldous-Shields model revisited \\[0ex] (with application
  to cellular ageing)
}

\thispagestyle{empty}

\author{{\sc by K. Best\thanks{Freiburg Initiative for Systems Biology
      and Albert-Ludwigs University, Freiburg}
    $\mbox{}^,$\thanks{Fakult\"at f\"ur Mathematik und Physik,
      Eckerstra\ss e 1, D-79104 Freiburg, Germany}$\;$, 
    P. Pfaffelhuber$\mbox{}^{\ast,\dagger,\,}$\thanks{Corresponding
      author; email: p.p@stochastik.uni-freiburg.de}}} \date{}

\maketitle
\unnumberedfootnote{\emph{AMS 2000 subject classification.} 60K35,
  92D20, 60J85, 05C05 .}

\unnumberedfootnote{\emph{Keywords and phrases.} Random tree, cellular
  senescence, telomere, Hayflick limit }

\vspace*{-6ex}

\begin{abstract}
  \noindent
  In Aldous and Shields (1988), a model for a rooted, growing random
  binary tree was presented. For some $c>0$, an external vertex splits
  at rate $c^{-i}$ (and becomes internal) if its distance from the
  root (depth) is $i$. For $c>1$, we reanalyse the tree profile, i.e.\
  the numbers of external vertices in depth $i=1,2,...$. Our main
  result are concrete formulas for the expectation and
  covariance-structure of the profile. In addition, we present the
  application of the model to cellular ageing. Here, we assume that
  nodes in depth $h+1$ are senescent, i.e.\ do not split. We obtain a
  limit result for the proportion of non-senescent vertices for large
  $h$.
\end{abstract}

\section{Introduction}
Trees arise in several applied sciences: In linguistics and biology,
trees describe the relationship of items (languages, species) and in
computer science, trees are used as data structures, e.g.\ for
sorting. Randomizing the input leads to random trees, which are object
of a large body of research. For applications in biology, see e.g.\
\cite{Berest09,Felsenstein2002}. Here, important examples are trees
arising from branching processes (e.g.\ Yule trees). In computer
science, prominent examples are search trees; see e.g.\
\cite{Neininger2004, Drmota2009}.

~

In this note, we are concerned with an application of random trees in
cellular biology. In the 1960s it was known that eukaryotic cells have
a limited replication capacity (\cite{pmid13905658}). The number of
generations until cells do not proliferate any more is today known as
the \emph{Hayflick limit} and the phenomenon that cells loose their
ability to proliferate is called \emph{cellular senescence}. The
molecular basis for cellular senescence were uncovered starting in the
1970s. A theory was developed which argued that during each round of
replication, the \emph{telomeres} (which are the end part of each
chromosome) are shortened due to physical constraints of the DNA
copying mechanism (\cite{pmid9415101}). In humans, these telomeres are
a multiple (i.e.\ more than 1000-fold) repetition of the base pairs
TTAGGG and up to 200 bases are lost in each replication round
(\cite{Levy:1992:J-Mol-Biol:1613801}). Most importantly, telomeres
have a stabilizing effect on the DNA. The \emph{DNA repair mechanism}
of a cell must be able to distinguish between usual DNA breaks (which
it is assumed to repair) and the telomeres (which it is assumed to
ignore). Hence, when telomeres become shorter this stabilizing effect
seizes and ageing occurs.  It can be observed that telomeres shrink
from 15 kilobases at birth to less than 5 kilobases during a lifetime
(\cite{pmid15471900}). However, the enzyme \emph{telomerase} is known
to be able to decrease the loss of telomeres during replication.  This
enzyme has been found to be active in stem cells and cancer cells,
which both are cell types with an (almost) unbounded replication
potential. The deeper understanding of the role of telomeres and
telomerase is an active field of research because of the medical
implications for ageing and cancer. In particular, it was awarded the
Nobel prize in medicine in 2009 (\cite{pmid19815741}).

~

We study the model of random trees introduced in Aldous and Shields in
\cite{AldousShields:1988:ProbTh} (hereafter referred to as [AS]) and
extend it for an application to cellular ageing. Given some $c>0$ and
a full binary tree $\mathbb T$, the model introduced in [AS] describes
the evolution of the vertices of the tree. Here, we distinguish
\emph{internal}, \emph{external} and \emph{prospective} vertices. At
$t=0$, the root is the only external vertex (and there are no internal
vertices). An external vertex $u\in\mathbb T$ in depth $|u|$ becomes
internal at rate $c^{-|u|}$. At the time it becomes internal, the two
daughter vertices in depth $|u|+1$ become external. We present our
result on the profile of the Aldous-Shields model in Theorem
\ref{Th1}.

For our application to cellular senescence, we will analyze a relative
of the Aldous-Shields model for $c>1$. Here, a critical depth $h$ is
fixed, and only external vertices in depth at most $h$ can become
internal. External vertices in depth $h+1$ never become
internal. Here, external vertices can be thought of as cells. The
depth of a vertex is the number of generations from the first
cell. Vertices in depth at most $h$ represent proliferating cells,
because they are able to produce offspring (i.e.\ daughter
cells). Vertices in depth $h+1$ represent senescent cells. This model
has two features, which appear to be realistic in cellular
senescence. First, the rate of cell proliferation decreases with the
generation of a cell, parameterized by $c>1$. Second, cells which have
already split too often loose their ability to proliferate at all. For
this model, we obtain a limit result for the frequency of
proliferating cells in Theorem \ref{T2}.

~

The paper is organized as follows: In Section \ref{sec:model}, we
state our results on the Aldous-Shields model. The application to
cellular senescence is carried out in Section \ref{sec:application},
where we also give an overview of other models for cellular senescence
in the literature. Section \ref{sec:proofs} contains the proofs for
our results on the Aldous-Shields model (Theorem \ref{Th1}), and in
Section \ref{sec::proofs2}, we give proofs for the results on the
model of cellular ageing (Theorem \ref{T2}).

\section{Model and results}\label{sec:model}
We start by introducing some notation. Let $\mathbb T$ be the complete
binary tree, given through
$$ \mathbb T = \bigcup_{n=0}^\infty \mathbb T_n$$
and
$$ \mathbb T_0 = \{\emptyset\}, \qquad \qquad \mathbb T_n = 
\{0,1\}^n\; \text{ for }n=1,2,...$$ We refer to elements in $\mathbb
T$ by \emph{vertices} and identify $u\in\mathbb T_n$ by a word of length
$n$ over the alphabet $\{0,1\}$, whose $i$th letter is $u_i$, $n\geq
1$. The vertex $\emptyset$ is the root of the tree and vertex
$u\in\mathbb T$ has two daughter vertices, $u0$ and $u1$. (We make the
convention that $\emptyset 0 := 0, \emptyset 1:=1$.) For $u\in\mathbb
T$ we set $|u|=n$ iff $u\in\mathbb T_n$.

We say that $u$ is an ancestor of $v$ if  $|u|<|v|$ and there
are $i_1,...,i_{|v|-|u|}\in\{0,1\}$ with $v= u i_1\cdots i_{|v|-|u|}$.
The ancestor induces a transitive order relation in $\mathbb T$, and we
write $u\prec v$ iff $u$ is ancestor of $v$.

\begin{definition}[Aldous-Shields model]
  Fix $c>0$. The (time-continuous) \emph{Aldous-Shields model} with
  parameter $c$ is a Markov jump process $\mathcal Y = (Y(t))_{t\geq
    0}$, $Y(t) = (Y_u(t))_{u\in\mathbb T}$ with state space
  $\{0,1\}^{\mathbb T}$, starting in $Y(0) =
  (\mathbbm{1}_{u=\emptyset})_{u\in\mathbb T}$. Given $Y(t) =
  y\in\{0,1\}^{\mathbb T}$ and $u\in\mathbb T$ with $y_u=1$, it jumps
  to $(\widetilde y_{v})_{v\in\mathbb T}$, given by
  $$ \widetilde y_v = \begin{cases} 0, & v=u,\\ 1, & v=u0 \text{ or } v=u1,\\ y_v, & \text{else,}\end{cases}$$
  at rate $c^{-|u|}$. In this case, we say that vertex $u$ splits.
\end{definition}

\begin{remark}[Internal and external vertices]
  Let $\mathcal Y = (Y(t))_{t\geq 0}$ be the Aldous-Shields model and
  $Y=Y(t)$ for some $t\geq 0$. It is important to note that the
  dynamics is such that any path $\emptyset, i_1, i_1
  i_2,...\in\mathbb T$ with $i_1, i_2,...\in\{0,1\}$, starting at the
  root, has exactly one element $u$ with $Y_u=1$. In particular, the
  sets
  $$ \{u: \exists v: u\prec v, Y_v=1\}, \qquad \{u: Y_u=1\}, \qquad \{u: \exists v: v\prec u, Y_v=1\}$$
  of internal, external and prospective vertices are disjunct.
\end{remark}

\begin{definition}[Profile]
  Let $\mathcal Y = (Y(t))_{t\geq 0}$ be the Aldous-Shields model and
  $Y=Y(t)$ for some $t\geq 0$. We define
  \begin{align}\label{eq:Xn}
    X_n := \sum_{u\in\mathbb T_n} Y_u, \qquad \widehat X_n := 2^{-n} X_n,    
  \end{align}
  the \emph{total number of external vertices} and the \emph{relative
    proportion of external vertices} in depth $n$, respectively. The
  vector $(X_n)_{n=0,1,2,...}$ is also called the \emph{profile} and
  \[X = \sum_{n=0}^\infty X_n\] is the \emph{total number of external
    vertices}.
\end{definition}

\begin{remark}[Dependence on $c$]
  The behaviour of the Aldous-Shields model strongly depends on $c$. A
  larger $c$ implies that the profile is more concentrated around
  certain depths. This is because a larger $c$ means that external
  vertices in smaller depth have a higher chance to be the next to
  split.  See Figure \ref{fig1} for an illustration.

  Two values of $c$ are of particular important in applications from
  computer science: for $c=1$, and if $X=n$, the set of external
  vertices is a binary search tree with $n$ external vertices. For
  $c=2$ and $X=n$, the set of external vertices is a digital search
  tree with $n$ external vertices; see e.g.\ \cite{Drmota2009}.
\end{remark}

\begin{figure}
  \begin{center}
    \includegraphics[width=7.1cm]{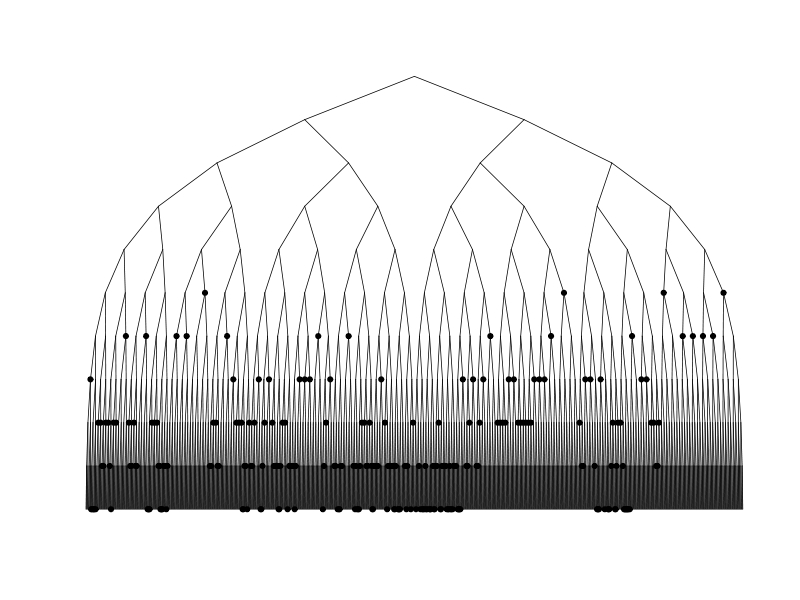}
    \includegraphics[width=7.1cm]{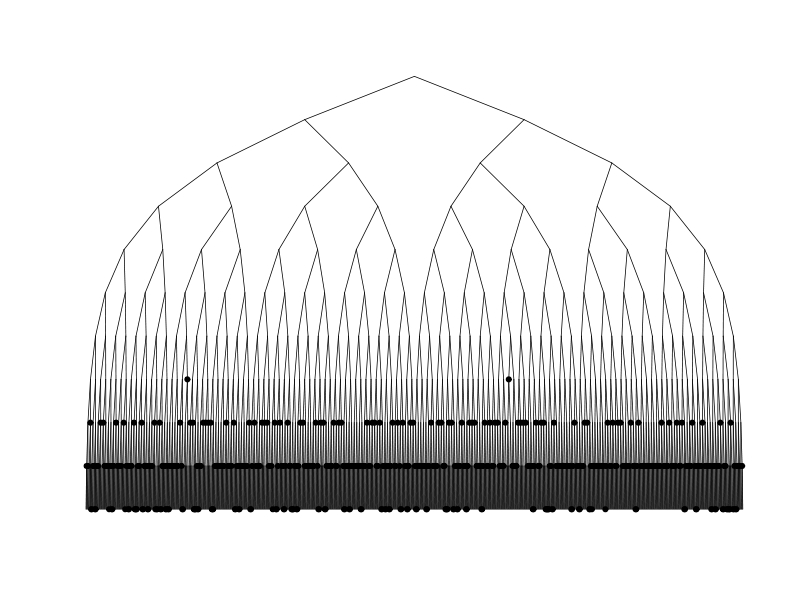}
  \end{center}
  \caption{\label{fig1}Two realizations of the Aldous-Shields model at
    the time when $500$ vertices are external for $c=1.05$ (A) and
    $c=3$ (B).  Only depths $0,...,10$ are drawn and the external
    vertices 
    are marked.}
\end{figure}

\begin{remark}[Relative frequencies]\label{rem:sumXn}
  We observe that 
  \begin{align}\label{eq:sumXn}
    \sum_{n=0}^\infty \widehat X_n(t) = 1
  \end{align}
  for all $t$, almost surely. To see this, note that $\widehat
  X_0(0)=1$ and $\widehat X_n(0)=0$ for $n>0$, i.e.\ \eqref{eq:sumXn}
  holds at $t=0$. Additionally, assume that some $u$ with $|u|=n$
  splits at time $t$. Then, we have that $\widehat X_n(t)-\widehat
  X_n(t-) = -2^{-n}$ and $\widehat X_{n+1}(t)-\widehat X_{n+1}(t-) =
  2\cdot 2^{-(n+1)} = 2^{-n}$. In particular, every split leaves
  $\sum_{n=0}^\infty \widehat X_n$ unchanged which shows
  \eqref{eq:sumXn}.
\end{remark}

\begin{remark}[Notation]
  In our results, we will give asymptotics of moments of
  $X_{n+i}(tc^n)$ for large $n$. Generally, for two sequences
  $(x_n)_{n=1,2,...}$ and $(y_n)_{n=1,2,...}$, which may depend on
  other parameters, we write
  \begin{align*}
    x_n \stackrel{n\to\infty}\approx y_n \qquad & \iff\qquad
    \lim_{n\to\infty} \frac{x_n}{y_n} = 1.
  \end{align*}
\end{remark}

\begin{theorem}[Moments of the profile and their limits]\label{Th1}
  Let $c>1$.  Define for $k\in\mathbb Z_+$
  \begin{align}
    \label{eq:akbk} a_k = (-1)^k \frac{c^k}{(c-1)\cdots(c^k-1)},
    \qquad\quad b_k := \frac{c^1\cdots c^k}{(c-1)\cdots (c^k-1)}
  \end{align}
  and $a_0=b_0=1$ (with the convention that an empty product equals
  1). For negative $k$, we set $a_k=b_k=0$. Moreover,
  $$ b_\infty := \Big(\prod_{l=1}^\infty (1-c^{-l})\Big)^{-1}
  \in (1;\infty).$$ Then, for $t>0$ and $i>-n$,
  \begin{align}\label{eq:Th1a}
    \mathbb E[X_{n+i}(tc^n)] = 2^{n+i} \cdot \sum_{k=0}^n a_k b_{n-k}
    e^{-c^{-i+k}t} \stackrel{n\to\infty}\approx b_\infty \cdot 2^{n+i}
    \cdot \sum_{k=0}^{\infty} a_{k} e^{-c^{-i+k}t}.
  \end{align}
  Moreover, for $i,i'\in\mathbb Z$, $i\leq i'$,
  \begin{equation}
    \label{eq:Th1b}
    \begin{aligned}
      \mathbb{COV}[ & X_{n+i}(tc^n), X_{n+i'}(tc^n)]
      \stackrel{n\to\infty}\approx a_{i,i'}\cdot
      \begin{cases}
        \displaystyle \frac{2}{2-c^2}\Big(\frac{2}{c}\Big)^{2n+i+i'}, & c<\sqrt{2} \\
        2^n n \sqrt{2}^{i+i'}, & c=\sqrt{2}, \\
        \displaystyle \displaystyle \frac{c^4}{2(c^2-1)(c^2-2)}
        2^{n+i'} c^{i-i'}, & c>\sqrt{2}.
    \end{cases}
    \end{aligned}
  \end{equation}
  where 
  \begin{align}\label{eq:aii}
    a_{i,i'} := b_\infty^2 \sum_{k,k'=0}^\infty a_{k} a_{k'}
    e^{-c^{-i+k}-c^{-i'+k'}t} c^{k+k'}.
  \end{align}
\end{theorem}

\begin{remark}[Convergence and covariance]
  \begin{enumerate}
  \item It is immediate from the Theorem that
    \begin{align}\label{eq:remcc1}
      \lim_{n\to\infty}\widehat X_{n+i}(tc^n) = b_\infty \cdot
      \sum_{k=0}^{\infty} a_{k} e^{-c^{-i+k}t}
    \end{align}
    in probability, for all $t>0$ and all $c>1$.
  \item The covariances given in \eqref{eq:Th1b} show a phase
    transition at $c=\sqrt 2$. Such a phase transition is already
    known from results by [AS] and \cite{DeanMajumdar2006}. However,
    these papers do not give explicit formulas for the covariance
    structure.
  \item Using \eqref{eq:231}, \eqref{eq:451} and \eqref{eq:sec3} it is
    also possible to obtain exact results for the covariance on the
    left hand side of \eqref{eq:Th1b}.
  \end{enumerate}
\end{remark}

\begin{remark}[Connection to results from Aldous and Shields (1988)]
  In [AS], the evolution of the vector $(\widehat
  X_n(t))_{n=0,1,2,...}$ is studied. In their Theorems, a law of large
  numbers if $X_n(t)$ to some deterministic limit $x_n(t)$ is stated
  and proved using martingale methods. Their result implies that
  \eqref{eq:remcc1} holds even almost surely on compact time
  intervals. In particular, they claim that the limit $x_i(t)$ of
  $\widehat X_{n+i}(tc^n)$ must satisfy $x_i(t) = x_{i+1}(ct)$, which
  clearly holds for the right hand side of \eqref{eq:remcc1}. In
  addition, they show that a suitably rescaled process,
  $2^{n/2}(\widehat X_{n+i}(t)-x_i(t))_{t\geq 0}$, converges as
  $n\to\infty$ weakly to a diffusion for $c>\sqrt 2$. The rescaling
  factor $2^{n/2}$ can also be seen from the above Theorem. Moreover,
  \eqref{eq:Th1b} shows that a convergence of $c^{n}(\widehat
  X_{n+i}(t)-x_i(t))_{t\geq 0}$ to a diffusion can be conjectured for
  $1< c <\sqrt 2$.
\end{remark}

\begin{remark}[Connection to work of Dean and Majumdar (2006)]
  In \cite{DeanMajumdar2006}, the total number of external vertices,
  $X$, was studied in the context of the Aldous-Shields model on an
  $m$-ary tree. In the binary case, a functional equation (their
  equation (2)) for the Laplace transform of $X(t)$ was shown to hold
  true. This equation uses the following fact: Given that $T$ is the
  random time of the first split in the model, it is clear that $T$ is
  mean one exponential and, in addition, $$ X(t) \stackrel d =
  \mathbbm{1}_{T\geq t} + \mathbbm{1}_{T\leq t} \big( X'\big(
  \tfrac{t-T}{c}\big) + X''\big( \tfrac{t-T}{c}\big)\big)$$ where $X'$
  and $X''$ are independent of $T$ and of each other and distributed
  like $X$.  From their identity on Laplace transforms,
  \cite{DeanMajumdar2006} show the phase transition for the variance
  of the number of occupied vertices at $c=\sqrt{2}$, which is also
  seen from Theorem \ref{Th1}.
\end{remark}

\section{Application: cellular ageing}\label{sec:application}
The first mathematical model for cellular senescence was given in
\cite{Levy:1992:J-Mol-Biol:1613801}. It takes several biological facts
into account. When DNA is copied, the double helix is unfolded and
both strands of DNA are copied. Only in one of the two strands there
are physical constraints by which the end of a chromosome cannot be
perfectly copied. This shortening of telomeres is independent for all
chromosomes.  In \cite{Levy:1992:J-Mol-Biol:1613801}, a fixed length
for telomeres which decreases by a fixed amount at each proliferation
event for one of the daughter cells and proliferation occurs along a
full binary tree is assumed. If the length of a telomere of one
chromosome falls below a threshold, a cell cannot replicate any more
and becomes senescent. This threshold takes the Hayflick limit into
account, which states that a cell line can only life for a limited
number of generations before it becomes senescent.

The model by \cite{Levy:1992:J-Mol-Biol:1613801} was extended in
several directions. A stochastic amount of loss of telomeres was
studied in \cite{Antal:2007:J-Theor-Biol:17631317}. In
\cite{Arino:1995:J-Theor-Biol:8551749} and
\cite{Olofsson:1999:Math-Biosci:10209937}, the binary tree of
proliferating cells from \cite{Levy:1992:J-Mol-Biol:1613801} was
replaced by a branching model. In particular,
\cite{Olofsson:1999:Math-Biosci:10209937} took cell death into
account, with different death rates above and below a critical
threshold of telomere length.  Age structure of cells (i.e.\ structure
which phase of the cell cycle) is taken into account by
\cite{Dyson:2007:J-Theor-Biol:17046024,
  Dyson:2002:Math-Biosci:11965249}. Moreover,
\cite{Arkus:2005:J-Theor-Biol:15833310} extend the model of
\cite{Levy:1992:J-Mol-Biol:1613801} by explicitly taking telomerase
activity (which is present in stem cells and cancer cells) into
account.

~

The idea to use the Aldous-Shields model for cellular ageing was
influenced by the following recent results:
\begin{enumerate}
\item In \cite{pmid11081503}, a model is proposed which distinguishes
  two states of telomeres: \emph{capped} and \emph{uncapped}. Only in
  the capped state, proliferation of the cell is still possible. In
  somatic cells, an uncapped telomere cannot be transformed to the
  capped state any more leading to senescent cells; see the model of
  \cite{Proctor:2003:Aging-Cell:12882407}. In stem and tumor cells,
  telomerase is (among other things) responsible for transitions from
  uncapped back to capped telomeres. Following
  \cite{RodriguezBrenes:2010:Proc-Natl-Acad-Sci-U-S-A:20207949}, the
  transition rate of the uncapped to the capped state in stem cell
  decreases with shorter telomeres.
\item In data, it has been observed that proliferating cells can behave
  differently. Motivated by data from \cite{baxter_study_2004,
    bonab_aging_2006, gupta_replicative_2007}, it is argued in
  \cite{Portugal:2008:Biosystems:18063293} that the rate of
  proliferation decreases for shorter telomeres. Their model produces
  a Gompertzian growth model which is known to fit to empirical data
  for somatic and tumor cells.
\end{enumerate}
In stem cells, the decreasing rate for an uncapped telomere to
reeneter the capped state for shorter telomeres from
\cite{RodriguezBrenes:2010:Proc-Natl-Acad-Sci-U-S-A:20207949} shows
exactly the behaviour of the Aldous-Shields model: cells with a long
replicative history proliferate slower. While
\cite{Portugal:2008:Biosystems:18063293} use a linear decrease in
replication rate, depending on telomere length, the Aldous-Shields
model uses a geometric decay of the proliferation rate.

Note that short telomeres can be seen as a form of \emph{damage}. In
\cite{EvansSteinsaltz:2007:ThPopBio}, models for cellular damage were
introduced. In their model, cells inherit damage to the daughter
cells. This model, as well as the Aldous-Shields model are among the
analytically tractable ones.


~

We state our model of cellular ageing:
\begin{definition}
  \label{def:Z}
  Fix $h\in\mathbb N, r>0$ and $c>1$ and let $\mathbb T^h :=
  \bigcup_{n=0}^{h+1} \mathbb T_n$. The process $\mathcal Z =
  (Z(t))_{t\geq 0}$, where $Z(t) = (Z_{u}(t))_{u\in\mathbb T^h}$ is a
  Markov jump process with state space $\{0,1\}^{\mathbb T^{h}}$,
  starting in $Z(0) = (\mathbbm 1_{u=\emptyset})_{u\in\mathbb
    T^h}$. Given $Z(t) = z\in\{0,1\}^{\mathbb T^h}$ and $u\in\mathbb
  T^h\setminus \mathbb T_{h+1}$ with $z_u=1$, it jumps to $(\widetilde
  z_{v})_{v\in\mathbb T}$, given by
  $$ \widetilde z_v = \begin{cases} 0, & v=u,\\ 1, & v=u0 \text{ or } v=u1,\\ 
    z_v, & \text{else,}\end{cases}$$ at rate $rc^{-|u|}$. (Note that
  vertices $u\in\mathbb T_{h+1}$ do not split.)
\end{definition}

Informally, every external vertex $u$ in this process represents a
cell. If $|u|=n$, we say that the cell is in generation $n$. The
process starts with a single mother cell. It proliferates at rate
$r$. All cells up to generation $h$ from the mother cell follow the
usual dynamics of the Aldous-Shields model (with time rescaled by a
factor of $r$), such that cells in generation $n$ proliferate at rate
$r\cdot c^{-n}$. If a cell is in generation $h+1$ from the mother
cell, its telomeres have reached the Hayflick limit and the cell is
not able to proliferate any more.

\begin{definition}[Relative frequency of proliferating cells]
  \label{def:L}
  In applications, the relative frequency of proliferating cells,
  \begin{align}\label{eq:L} 
    L(t) := \frac{Z^p(t)}{Z^p(t) + Z^s(t)}
  \end{align}
  with
  $$ Z^p(t) := \sum_{u\in \mathbb T^h\setminus \mathbb T_{h+1}} Z_{u}(t), \qquad 
  Z^s(t) := \sum_{u\in\mathbb T_{h+1}} Z_{u}(t),$$ is of particular
  importance. Here, $Z^p(t)$ and $Z^s(t)$ is the number of
  proliferating and senescent cells at time $t$, respectively.
\end{definition}

\begin{theorem}[Frequency of replicating cells]
  \label{T2}
  For $\mathcal Z$ and $L$ as in Definition \ref{def:Z} and
  \ref{def:L},
  \begin{align}
    \lim_{h\to\infty} L(tc^h) = \frac{\sum_{i=0}^\infty
      \sum_{k=0}^\infty 2^{-i} a_k e^{-c^{i+k}t/r}}{ \sum_{i=0}^\infty
      \sum_{k=0}^\infty a_k(2^{-i}e^{-c^{i+k}t/r} +
      2e^{-c^{-i-1+k}t/r})} \label{eq:T2}
  \end{align}
  in probability, for all $t>0$.
\end{theorem}

\begin{remark}[Simulations]
  In our model for cellular senescence, Theorem \ref{T2} describes the
  decrease in the frequency of proliferating cells.  This frequency
  has been measured empirically; see e.g.\ \cite[Figure
  5]{Arkus:2005:J-Theor-Biol:15833310} and \cite[Figure
  2]{Arino:1995:J-Theor-Biol:8551749}. As can be seen from the
  Theorem, every $c$ gives a specific curve of decrease; see also
  Figure \ref{fig3}.  These curves can be fit to data in order to
  estimate $c$. As the figure shows, the limiting result of Theorem
  \ref{T2} already gives a good fit for simulations which use $h=20$.
\end{remark}

\begin{figure}
  \begin{center}
    \includegraphics[width=10cm]{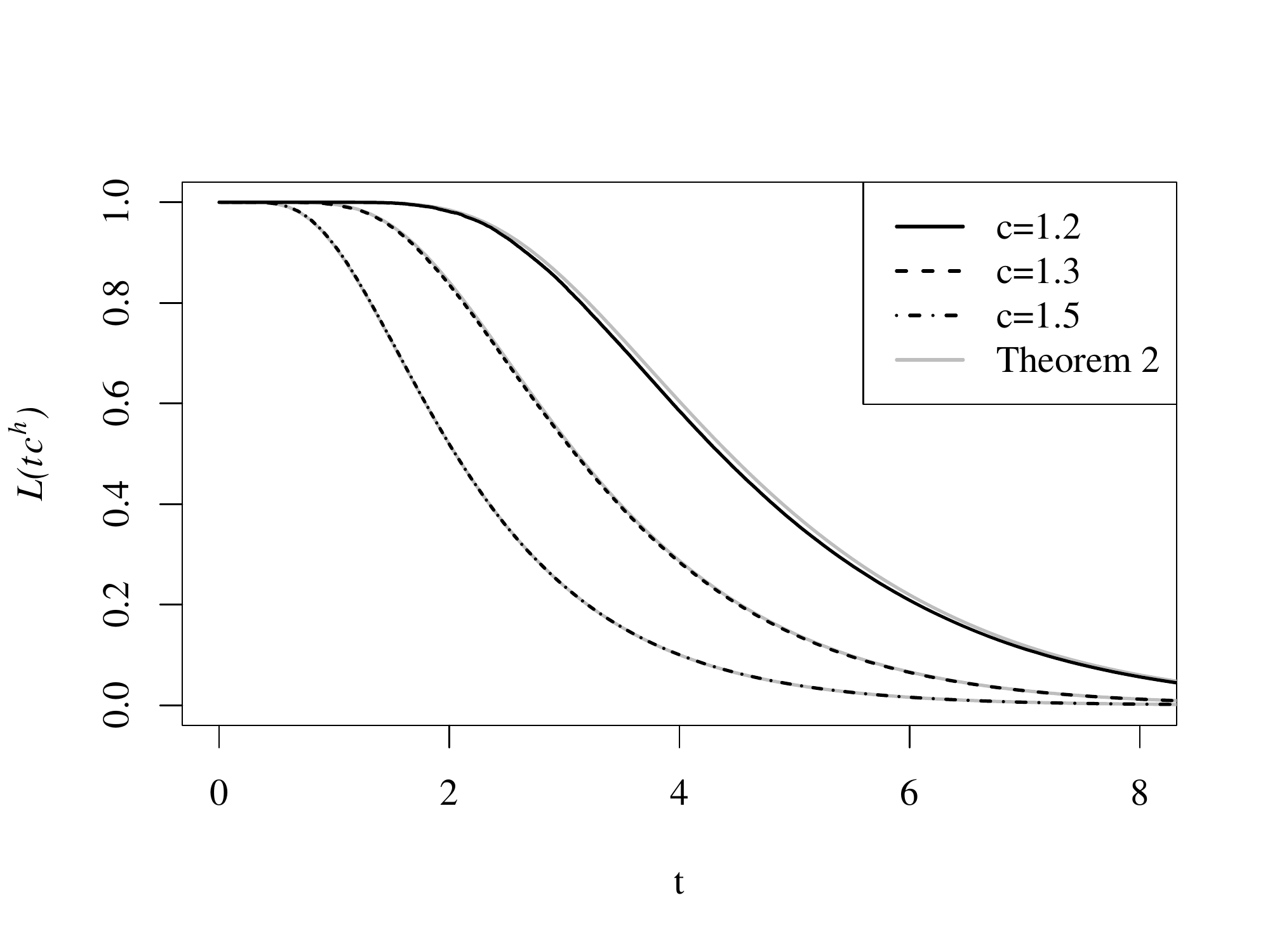}
  \end{center}
  \caption{\label{fig3}The decrease of the frequency of proliferating
    cells strongly depends on $c$. The figure shows simulations for
    $c=1.2, 1.3, 1.5$ and $h=20$. The grey lines show the limit result
    \eqref{eq:T2} from Theorem 2 with $r=1$.}
\end{figure}

\section{Proof of Theorem \ref{Th1}}\label{sec:proofs}

\subsection{Preliminaries}
The key ingredient in the proof is the quantity
\begin{align}\label{eq:zn}
  y_n(t) := \mathbb E[Y_{0_n}(t)] \text{ with }
  0_n:=\underbrace{0\cdots 0}_{n \text{ times}}.
\end{align}
Note that the dynamics of $(Y_{0_n})_{n=0,1,2,...}$ (taking values in
$\{0,1\}^{\{0,1,2,...\}}$ with only one entry being 1) is
autonomous. It is given by the following rule: if $Y_{0_n}(t)=1$,
then, at rate $c^{-n}$ a transition occurs to the configuration
$Y_{0_n}(t)=0, Y_{0_{n+1}}(t)=1$. From the dynamics of the
Aldous-Shields model, it is clear that $\underline y(t) =
(y_n(t))_{n=0,1,2,...}$ follows the differential equations
\begin{align}
  \dot y_n = c^{-(n-1)} y_{n-1} - c^{-n} y_n, \qquad n\geq
  0 \label{eq:AS21}
\end{align}
with $y_{-1}=0$; compare (2.1) in [AS]. We rewrite the equation to
obtain
$$ \dot{\underline y} = A \underline y$$
with
$$ A = \left( \begin{matrix} 
    -1 &          &         &         &         &         & \\ 
    1  & - c^{-1} &         &         &         &         & \\ 
    & c^{-1}   & -c^{-2} &         &         &         &  \\ 
    &          & c^{-2}  & -c^{-3} &         &         &  \\ 
    &          &         & c^{-3}  & -c^{-4} &         &  \\ 
    &          &         &         & c^{-4}  & -c^{-5} &  \\ 
    &          &         &         &         & \cdots  & \cdots \\ 
  \end{matrix}\right)$$

\noindent
Our first Lemma provides essential facts about the matrix $A$. Recall
$a_k$ and $b_k$ from \eqref{eq:akbk}.

\begin{lemma}\label{l:key}
  Let $E=(e_{ij})_{i,j=0,1,2,...}$ and $F=(f_{ij})_{i,j=0,1,2,...}$
  be given by
  $$ e_{ij} = a_{i-j}, \qquad\qquad f_{ij}=b_{i-j}.$$

  The matrix $E$ contains the eigenvectors of $A$ to the eigenvalues
  $\lambda_i = -c^{-i}, i=0,1,2,...$ and $E$ and $F$ are inverse
  to each other.
\end{lemma}

\begin{proof}
  To see that $E$ contains the eigenvectors of $A$, note that for
  $i\geq j$
  $$ (AE)_{ij} = c^{-i+1}a_{i-1-j} - c^{-i} a_{i-j} = (c^{-i+1}\frac{1-c^{i-j}}{c} - 
  c^{-i})\, a_{i-j} = -c^{-j} a_{i-j}.$$ To see that $E$ and $F$ are
  inverse to each other, it follows from the definition of $a_k$ and
  $b_k$ that
  $$(F E)_{ii} = a_0\cdot b_0=1\qquad \text{and}\qquad 
  (F E)_{ij} = 0 \text{ for }i<j.$$ For $i>j$, we set $n:=i-j>0$ and
  obtain
  \begin{align*}
    (F E)_{i, j} & = \sum_{k=j}^i b_{i-k} a_{k-j} = \sum_{k=j}^{n+j}
    b_{j+n-k} a_{k-j}= \sum_{k=0}^{n} b_{n-k} a_{k} \\ & =
    \sum_{k=0}^n (-1)^k \frac{c\cdots c^{n-k}\cdot c^k}{ (c-1)\cdots
      (c^{n-k}-1)\cdot (c-1)\cdots (c^k-1) } \\ & =
    \frac{c^n}{(c-1)\cdots (c^n-1)} \sum_{k=0}^n (-1)^k
    \frac{c^{1+\cdots +(n-k-1)}}{(c-1)\cdots (c^{n-k}-1)}
    (c^{k+1}-1)\cdots (c^n-1) \\ & = \frac{(-c)^n}{(c-1)\cdots
      (c^n-1)} \sum_{k=0}^n (-1)^k\frac{c^{\binom{k}{2}}}{(c-1)\cdots
      (c^k-1)}(c^{n-k+1}-1)\cdots (c^n-1),
  \end{align*}
  where we have reversed the order of the summands in the last equality.
  We rewrite $(c^{l}-1) \cdots (c^{l'}-1) = \prod\limits_{j=l}^{l'}(c^j-1)=1$ 
  for $l>l'$ and  
  claim that for
  all $n>0$
  \begin{align}\label{eq:claim1}
    \sum_{k=0}^n (-1)^k\frac{c^{\binom{k}{2}}}{(c-1)\cdots
      (c^k-1)}(c^{n-k+1}-1)\cdots (c^n-1) = 0
  \end{align}
  which implies that $F$ and $E$ are inverse to each other. We use
  induction and note that the assertion is clear for $n=1$. Given it
  is true for $n$, we have
  \begin{align*}
    \sum_{k=0}^{n+1} & (-1)^k \frac{c^{\binom{k}{2}}}{(c-1)\cdots
      (c^k-1)}(c^{n-k+2}-1)\cdots (c^{n+1}-1)\\ 
      & = 1 +
    \sum_{k=1}^{n+1}(-1)^k \frac{c^{\binom{k}{2}}}{(c-1)\cdots
      (c^k-1)} 
      (c^{n-k+2}-1)\cdots (c^n-1) 
      [(c^{n+1} - c^{n+1-k}) + (c^{n-k+1}-1)]\\ 
      & = \sum_{k=1}^{n+1}
    (-1)^k\frac{c^{\binom{k}{2}}}{(c-1)\cdots (c^k-1)}
    c^{n+1-k}(c^k-1) (c^{n-k+2}-1)\cdots(c^n-1) \\ 
      & = \sum_{k=1}^{n+1}
    (-1)(-1)^{k-1}\frac{c^{\binom{k-1}{2}}c^k}{(c-1)\cdots (c^{k-1}-1)(c^k-1)}
    c^{n+1}c^{-k}(c^k-1) (c^{n-k+2}-1)\cdots(c^n-1) \\ 
      & = -c^{n+1}
    \sum_{k=1}^{n+1} (-1)^{k-1} \frac{c^{\binom{k-1}{2}}}{(c-1)\cdots
      (c^{k-1}-1)} (c^{n-(k-1)+1}-1)\cdots (c^n-1) = 0
  \end{align*}  
  where we have used the induction hypothesis in the second and in the
  last equality. Hence, we have shown \eqref{eq:claim1} and the proof
  is complete.
\end{proof}

\subsection{First order structure; proof of \eqref{eq:Th1a}}
By linearity, we can now explicitly solve \eqref{eq:AS21} using Lemma
\ref{l:key}. Since $E$ contains the eigenvalues of $A$ and $F$ is
inverse to $E$, we immediately write, using $\underline y(t) :=
(y_0(t),y_1(t), y_2(t),...)$ and $D:=\text{diag}(-1,-c^{-1},
-c^{-2},...)$
\begin{align*}
  \big(2^{-n}\mathbb E[X_{n}(t)]\big)_{n=0,1,2,...} & = y(t) = e^{At}
  \underline y(0)^\dagger = E e^{Dt} F \underline y(0)^\dagger =
  Ee^{Dt} (b_0, b_1,...)^\dagger \\ & = \Big(\sum_{k=0}^{n}
  a_{n-k} b_{k} e^{-c^{-k}t} \Big)_{n=0,1,2,...}
\end{align*}
because $z(0)=(1,0,0,...)$ as the process starts with
$(Y(0))=(\mathbbm 1_{u=\emptyset})_{u\in\mathbb T}$.  We have shown the first
part of \eqref{eq:Th1a} and in order to prove the second part, fix
$i,k,t$ and note that $n\mapsto |a_k| b_{n+i-k} e^{-c^{-i+k}t}$ is
increasing with a summable limit. Hence, by dominated convergence,
\begin{align*}
  y_{n+i}(tc^n) & = \sum_{k=0}^{\infty} a_{k}b_{n+i-k}
  e^{-c^{-n-i+k}tc^n} \stackrel{n\to\infty}\approx b_\infty
  \sum_{k=0}^{\infty} a_{k} e^{-c^{-i+k}t}.
\end{align*}

\subsection{Second order structure; proof of \eqref{eq:Th1b}}
\noindent
Now we come to the second order structure. Similar to the definition
of $y_n$ in \eqref{eq:zn}, we set for $n\leq n'$
$$ y_{n,n',m}(t) := \mathbb{COV}[Y_{0_n}(t), Y_{0_m1_{n'-m}}(t)]\text{ with } 0_m1_{n'-m}:= 
\underbrace{0\cdots 0}_{m\text{ times}}\underbrace{1\cdots
  1}_{n'-m\text{ times}}.$$ In order to see the connection of
$y_{n,n',m}(t)$ and $\mathbb{COV}[X_n(t), X_{n'}(t)]$, we define
the depth of the most recent common ancestor of $u, u'$ as
$$M_{u, u'} :=
\sup\{|v|: v\preceq u, v\preceq u'\}, \qquad u, u'\in\mathbb T,$$ 
where $v\preceq u$ if $v\prec u$ or $v=u$. 
Let $U, U'$
be two random variables, where $U$ is uniformly distributed on
$\mathbb T_n$ and $U'$ uniformly distributed on $\mathbb T_{n'}$,
independent of all the rest. The distribution of $M_{U,U'}$ is given
by (recall $n\leq n'$)
$$ \mathbb P[M_{U,U'} = m] = 2^{-((m+1)\wedge n)}, \qquad m=0,...,n.$$
We write
\begin{equation}\label{eq:231}
  \begin{aligned}
    \mathbb{COV} [X_{n} & (t), X_{n'}(t)] = \sum_{u\in\mathbb
      T_n}\sum_{u'\in\mathbb T_{n'}} \big(\mathbb E[Y_{u}(t)Y_{u'}(t)]
    - \mathbb E[Y_{u}(t)]\cdot \mathbb E[Y_{u'}(t)]\big)
    \\
    & = 2^{n+n'} \mathbb{COV}[Y_{U}(t),Y_{U'}(t)] \\ & =
    2^{n+n'} \big( \mathbb E\big[
    \mathbb{COV}[Y_{U}(t),Y_{U'}(t)|M_{U,U'}]\big] +
    \mathbb{COV}\big[\mathbb E[Y_U|M_{U,U'}], \mathbb
    E[Y_{U'}|M_{U,U'}]\big]\big) \\ & = 2^{n+n'} \mathbb E\big[
    \mathbb{COV}[Y_{U}(t),Y_{U'}(t)|M_{U,U'}]\big] \\ & = 2^{n+n'}
    \sum_{m=0}^{n} 2^{-((m+1)\wedge n)} y_{n, n', m}(t).
  \end{aligned}
\end{equation}
The second to last equality holds as $\mathbb{E}[Y_{U}|M_{U,U'}] =
\mathbb E[Y_U]$ and $\mathbb{E}[Y_{U'}|M_{U,U'}]=\mathbb
E[Y_{U'}]$. In order to use the last expression, note that for
$y_{n,n',m}$ the vertices $0_n$ and $0_m1_{n'-m}$ have the vertex
$0_m$ as their most recent common ancestor and so $M_{0_n,
  0_m1_{n'-m}}=m$. We let $T_{m}$ be the last time $Y_{0_m}$ is
external, respectively the time when $Y_{0_{m+1}}$ becomes external,
i.e.\ $T_{m}$ is the sum of exponentials with parameters $1,
c^{-1},c^{-2},...,c^{-m}$. Hence,
\begin{align}\label{eq:lapl1}
  \mathbb E[e^{-\lambda T_{m}}] & = \prod_{l=1}^m
  \frac{c^{-l}}{c^{-l}+\lambda} = \prod_{l=1}^m \frac{1}{1+\lambda
    c^l}.
\end{align}

~

Using the last equations we now prove \eqref{eq:Th1b} in three
steps. First, we give a representation of $y_{n,n',m}$ in terms of a
functional of $T_m$. Second, we derive the asymptotics of
$y_{n+i,n+i',m}(tc^n)$ for large $n$ using this representation. Last,
we plug this asymptotics into \eqref{eq:231}.

\begin{step}[Exact representation of $y_{n,n',m}(t)$]
  In this step we show that for $m < n\leq n'$
  \begin{equation}\label{eq:451}
  \begin{aligned}
    y_{n,n',m}(t) & = \sum_{k=0}^{n-m}\sum_{k'=0}^{n'-m} a_{k} a_{k'} b_{n-m-k}
    b_{n'-m-k'} \\ &\qquad \qquad \qquad
    \cdot\mathbb{COV}[e^{-c^{-n+k}(t-T_m)}\mathbbm{1}_{t\geq T_m},
    e^{-c^{-n'+k'}(t-T_m)} \mathbbm{1}_{t\geq T_m}] 
  \end{aligned}
  \end{equation}
  and for $m=n\leq n'$
  \begin{equation}   \label{eq:452}
     y_{n,n',m}(t)  = \delta_{n,n'} y_n(t) - y_n(t)  y_{n'}(t),
  \end{equation}
  where $\delta_{n,n'}$ is Kronecker's $\delta$.
  \begin{proof}
    For \eqref{eq:451}, observe that $Y_{0_n}(t)$ and
    $Y_{0_m1_{n'-m}}(t)$ are independent, given $T_{m}$. Moreover, it
    is clear that $\mathbb E[Y_{0_n}(t)|T_m] = y_{n-m}(c^{-m}(t-T_m))
    \mathbbm 1_{t\geq T_m}$. Hence, by \eqref{eq:Th1a},
    \begin{align*}
      &\mathbb{COV}[Y_{0_n}(t), Y_{0_m1_{n'-m}}(t)] \\ & = \mathbb
      E[\mathbb{COV}[Y_{0_n}(t), Y_{0_m1_{n'-m}}(t)| T_{m}]] \\ &
      \qquad \qquad \qquad \qquad \qquad \qquad \qquad + \mathbb{COV}[
      \mathbb E[Y_{0_n}(t)|T_{m}], \mathbb
      E[Y_{0_m1_{n'-m}}(t)|T_{m}]] \\ & = \mathbb{COV}[
      y_{n-m}(c^{-m}(t-T_{m}))\mathbbm 1_{t\geq
        T_m},y_{n'-m}(c^{-m}(t-T_{m}))\mathbbm 1_{t\geq T_m}] \\ & =
      \sum_{k=0}^{n-m} \sum_{k'=0}^{n'-m} a_{k} a_{k'} b_{n-m-k}
      b_{n'-m-k'} \mathbb{COV}[e^{-c^{-n+k}(t-T_m)}\mathbbm 1_{t\geq
        T_m}, e^{-c^{-n'+k'}(t-T_m)}\mathbbm 1_{t\geq T_m}].
    \end{align*}
    For \eqref{eq:452}, note that $m=n$ implies that
    $Y_{0_n}(t)Y_{0_m1_{n'-m}}(t) = \delta_{n,n'}Y_{0_n}(t)$ and the
    result follows.
  \end{proof}
\end{step}

\begin{step}[Asymptotics of $y_{n+i,n+i',m}(tc^n)$ for large $n$]
  The aim of this step is to establish that for $i,i'\in\mathbb Z$ and
  $m\in\mathbb Z_+$,
  \begin{align}\label{eq:453}
    y_{n+i,n+i',m}(tc^n) & \stackrel{n\to\infty}\approx b_\infty^2
    \sum_{k=0}^\infty \sum_{k'=0}^\infty a_{k} a_{k'}
    e^{-t(c^{-i+k}+c^{-i'+k'})} c^{-2n-i-i'+k+k'}
    \frac{c^{2(m+1)}-1}{c^2-1}.
  \end{align}
  \begin{proof}
    For large $n$, we have $tc^n>T_m$ and so, by \eqref{eq:451},
    \eqref{eq:452} and dominated convergence,
    \begin{equation}
      \label{eq:sec2}
      \begin{aligned}
        y_{n+i,n+i',m}(tc^n) & \stackrel{n\to\infty}\approx b_\infty^2
        \sum_{k=0}^\infty \sum_{k'=0}^\infty a_{k}
        a_{k'}\mathbb{COV}[e^{-c^{-n-i+k}(tc^n-T_m)},
        e^{-c^{-n-i'+k'}(tc^n-T_m)}] \\ & = b_\infty^2
        \sum_{k=0}^\infty \sum_{k'=0}^\infty a_{k} a_{k'}
        e^{-t(c^{-i+k}+c^{-i'+k'})}\mathbb{COV}[e^{c^{-n-i+k}T_m},
        e^{c^{-n-i'+k'}T_m}].
      \end{aligned}
    \end{equation}
    Now, for all $j,j'\in\mathbb Z$, using \eqref{eq:lapl1}, for $n$
    large enough,
    \begin{equation}
      \label{eq:sec3}
      \begin{aligned}
        \mathbb{COV}[ & e^{c^{-n-j}T_m},e^{c^{-n-j'}T_m}] \\ & =
        \prod_{l=0}^m \frac{1}{1 - c^{-n-j+l} - c^{-n-j'+l}} -
        \prod_{l=0}^m \frac{1}{1 - c^{-n-j+l} - c^{-n-j'+l} +
          c^{-2n-j-j'+2l}} \\ & \stackrel{n\to\infty}\approx
        \prod_{l=0}^m (1 + c^{-n-j+l} +
        c^{-n-j'+l} + c^{-2n-2j+2l} + 2c^{-2n-j-j'+2l} + c^{-2n-2j'+2l} ) \\
        & \qquad - \prod_{l=0}^m (1 + c^{-n-j+l} + c^{-n-j'+l} +
        c^{-2n-2j+2l} + c^{-2n-j-j'+2l} + c^{-2n-2j'+2l} ) \\ &
        \stackrel{n\to\infty}\approx \sum_{l=0}^m c^{-2n-j-j'+2l} \\ &
        = c^{-2n-j-j'} \frac{c^{2(m+1)}-1}{c^2-1}.
      \end{aligned}
    \end{equation}
    Plugging the last expression into \eqref{eq:sec2} gives
    \eqref{eq:453}.
  \end{proof}
\end{step}

\begin{step}[Combining \eqref{eq:453} and \eqref{eq:231}]
  We write immediately, using $j=i-k$ and $j'=i'-k'$,
  \begin{equation}
    \begin{aligned}
      \mathbb{COV}& [X_{n+i}(tc^n), X_{n+i'}(tc^n)] \\ &
      \stackrel{n\to\infty}\approx b_\infty^2
      \Big(\frac{2}{c}\Big)^{2n+i+i'} \sum_{m=1}^{n+i+1}
      \sum_{k,k'=0}^\infty a_{k} a_{k'} e^{-t(c^{-i+k}+c^{-i'+k'})}
      c^{k+k'} 2^{-m}\frac{c^{2m}-1}{c^2-1}.
    \end{aligned}\label{eq:1}
  \end{equation}
  Noting that $\sum_{k=0}^\infty |a_k| c^k<\infty$, we see that for
  $a_{i,i'}(t)$ given by \eqref{eq:aii}
  \begin{align*}
    \mathbb{COV}& [X_{n+i}(tc^n), X_{n+i'}(tc^n)]
    \stackrel{n\to\infty} \approx a_{i,i'}(t)
    \Big(\frac{2}{c}\Big)^{2n+i+i'}\sum_{m=1}^{n+i+1}
    \frac{2^{-m}(c^{2m}-1)}{c^2-1}\\
    & \stackrel{n\to\infty}\approx a_{i,i'}(t) \cdot \begin{cases}
      \displaystyle \frac{2}{2-c^2}\Big(\frac{2}{c}\Big)^{2n+i+i'}, & c<\sqrt{2} \\
      2^n n \sqrt{2}^{i+i'}, & c=\sqrt{2}, \\
      \displaystyle \frac{c^4}{2(c^2-1)(c^2-2)} 2^{n+i'} c^{i-i'}, &
      c>\sqrt{2}.
    \end{cases}
  \end{align*}
  This finally shows \eqref{eq:Th1b} and finishes the proof of Theorem
  \ref{Th1}.
\end{step}
\hfill\qed

\section{Proof of Theorem \ref{T2}}
\label{sec::proofs2}
The parameter $r$ is only a rescaling of time. Hence, we can safely
assume $r=1$ in our proof.  Let $\mathcal Y = (Y(t))_{t\geq 0}$ be the
Aldous-Shields model with parameter $c$ and $X_n(t)$ as in
\eqref{eq:Xn}. Defining
$$ \widetilde X_{h+1}(t) := \sum_{n=h+1}^\infty 
2^{h+1-n} X_{n}(t),$$ it is important to note that 
$$\widetilde X_{h+1}(t) = \# \{u\in\mathbb T_{h+1}: \; \exists v: u\preceq v, 
Y_v(t)=1\},$$ almost surely; see also Remark \ref{rem:sumXn}. This
implies that we can couple $\mathcal Y$ and $\mathcal Z$ in the sense
that
$$(X_1(t), ..., X_h(t), \widetilde X_{h+1}(t))_{t\geq 0} \stackrel
d = \Big( \sum_{u\in\mathbb T_1} Z_u(t),..., \sum_{u\in\mathbb
  T_{h+1}} Z_u(t)\Big)_{t\geq 0}.$$ By Theorem \ref{Th1}, for
$n\leq h$ and $i=0,1,2,...$
\begin{equation*}
  \label{eq:T21}
  \begin{aligned}
    Z^p(tc^h) & \stackrel d = \sum_{i=0}^h X_i(tc^h) = \sum_{i=0}^h
    X_{h-i}(tc^h) \stackrel{h\to\infty}\approx b_\infty \cdot 2^{h}
    \sum_{i=0}^\infty \sum_{k=0}^\infty 2^{-i} a_k e^{-c^{i+k}t},\\
    Z^s(tc^h) & \stackrel d =\widetilde X_{h+1}(tc^h) =
    \sum_{i=1}^\infty 2^{-i+1} X_{h+i}(tc^h)
    \stackrel{h\to\infty}\approx b_\infty \cdot 2^{h}
    \sum_{i=1}^\infty \sum_{k=0}^\infty 2a_k e^{-c^{-i+k} t}
  \end{aligned}
\end{equation*}
in probability, for all $t>0$. Using the last two limits in the
definition of $L$ in \eqref{eq:L} gives the result.
\hfill\qed

\subsubsection*{Acknowledgements}
This research was supported by the BMBF through FRISYS (Kennzeichen
0313921).


\end{document}